\documentclass[reqno,11pt]{amsart}
\usepackage{amsmath,amssymb,amsthm,amsfonts,mathrsfs,latexsym,times,color}

\usepackage{amsmath,amsfonts,mathtools}
\usepackage{amsthm}
\usepackage{graphicx}
\usepackage{color,soul}
\usepackage{multicol}

\newcommand\blfootnote[1]{%
\begingroup
\renewcommand\thefootnote{}\footnote{#1}%
\addtocounter{footnote}{-1}%
\endgroup}

\newtheorem{thm}{Theorem}[section]

\newtheorem{notation}{Notation}[section]
\newtheorem{lemma}[thm]{Lemma}
\newtheorem{remark}{Remark}[section]
\newtheorem{theorem}[thm]{Theorem}
\newtheorem{proposition}[thm]{Proposition}

\numberwithin{equation}{section}

\usepackage{epsfig}
\usepackage{graphicx}
\usepackage{colordvi}
\usepackage{graphics}
\usepackage{color}
\usepackage{multirow}

\begin{document}

\title[exact steady states of a chemotaxis model]{Unveiling Explicit Patterns: Exact Steady States and Stability in a Confined Chemotaxis Model}

\author[Y. Huang]{Yue Huang}
\address[Y. Huang]{School of Mathematical Sciences, Harbin Engineering University, Harbin, 150001, China}
\email{huangyue@hrbeu.edu.cn}

\author[L. Xue]{Ling Xue}
\address[L. Xue]{School of Mathematical Sciences, Harbin Engineering University, Harbin, 150001, China}
\email{lxue@hrbeu.edu.cn}

\author[K. Zhao]{Kun Zhao}
\address[K. Zhao]{School of Mathematical Sciences, Harbin Engineering University, Harbin, 150001, China}
\email{kzhao@hrbeu.edu.cn}

\author[X. Zheng]{Xiaoming Zheng}
\address[X. Zheng]{Department of Mathematics, Central Michigan University, Mount Pleasant, MI 48859, USA}
\email{zheng1x@cmich.edu}



\begin{abstract}
Inspired by Carrillo-Li-Wang's work [Proc. London Math. Soc., 2021] on stationary solutions to the singular Keller-Segel system, this paper presents a novel family of explicit steady-state solutions for the same model on a bounded interval, expressed in terms of trigonometric and hyperbolic functions. Under Dirichlet boundary conditions and within a biologically stable parameter regime, these solutions, including singular types such as secant and cosecant, are rigorously derived and analyzed. Their stability is established via energy methods, yielding precise thresholds for pattern persistence. These results provide valuable benchmarks for numerical validation and offer insights into boundary-driven pattern formation.
\end{abstract}

\maketitle

\section{Introduction}

Since Newton’s inception of differential equations, the search for explicit solutions has constituted a central mathematical challenge and a vital tool for interpreting applied scientific problems. This is especially evident in the study of nonlinear reaction-diffusion systems, where explicit solutions serve as key benchmarks for numerical methods and yield direct insight into pattern formation mechanisms. The Keller-Segel model is particularly notable: it captures striking examples of biological self-organization, from bacterial band formation to tumor-induced angiogenesis, yet presents deep mathematical challenges due to its strongly nonlinear structure. In this paper, we construct a new family of explicit steady-state solutions for a Keller-Segel-type model, expressed in terms of trigonometric and hyperbolic functions, and accompany them with a rigorous stability analysis. These solutions serve a dual purpose: they provide exact benchmarks for validating computational algorithms in spatially confined domains, and they allow precise analytical identification of stability thresholds that govern pattern selection.

\subsection{Model and Relevance}

We consider the system:
\begin{subequations}\label{model}
\begin{alignat}{2}
u_t &= \mathsf{d} u_{xx} - \chi [u(\ln c)_x]_x, \quad &x&\in\mathbb{R},\ t>0, \label{model-1} \\
c_t &= \varepsilon c_{xx} - \sigma c - \mu u c^\mathsf{m}, \quad &x&\in\mathbb{R},\ t>0, \label{model-2}
\end{alignat}
\end{subequations}
with the following physical interpretation:
\begin{itemize}
\item $u(x,t)$ is the cell density, and $c(x,t)$ is the chemical attractant concentration;
\item parameters $\mathsf{d}$, $\chi$, $\varepsilon$, $\sigma$, $\mu$, and $\mathsf{m}$ represent diffusion rate, chemotactic sensitivity, chemical diffusivity, degradation/production rate, source strength, and nonlinear kinetic exponent, respectively.
\end{itemize}
This generalized framework includes several important applied models:
\begin{itemize}
\item Classical Keller-Segel ($\varepsilon\geqslant 0$, $0\leqslant \mathsf{m}<1$): models bacterial band formation in response to oxygen gradients \cite{KS3}.
\item Othmer-Stevens \& Levine-Sleeman ($\varepsilon=0$, $\mathsf{m}=1$): describes nondiffusive chemical signaling in cellular communication \cite{LS,OS}.
\item Levine-Nilsen-Hamilton-Sleeman ($\varepsilon>0$, $\mathsf{m}=1$): applies to tumor angiogenesis and protease inhibitor dynamics \cite{LSN}.
\end{itemize}

The logarithmic sensitivity term follows the Weber-Fechner law of perceptual response \cite{WF}, which reflects biological sensing mechanisms but introduces mathematical challenges due to its singular nature. This nonlinearity allows the model to capture essential phenomena such as wave propagation and pattern formation, yet it complicates both analytical treatment and numerical approximation.

\subsection{Development and Motivation}

Mathematical analysis of the chemotaxis system \eqref{model} with $\mathsf{m}=1$ has advanced considerably since the early 2000s, beginning with foundational work by Fontelos-Friedman-Hu on global well-posedness \cite{FFH}. Subsequent developments include:
\begin{itemize}
\item global existence and long-time behavior analysis \cite{Spike,FTWZZ,FXXZ,FZZ,FZ,L1,GXZZ,1d1,1d2,DL,1d3,MD3,1d4,1d6,1d7,ZZ,ZZMZ,ZLWZ},
\item boundary layer formation in diffusivity limits \cite{Spike,FTWZZ,FZZ,DL3,DL2,DL1,1d5},
\item characterization of traveling waves and steady states \cite{mTW-1,mTW-2,TW1,TW2,TW3,TW4,TW5,TW6,TW7,TW8,TW9}.
\end{itemize}
In contrast, the multi-dimensional case, featuring energy criticality and supercriticality, poses greater difficulties, and progress there has been more limited. To the best of our knowledge, existing results address global dynamics and diffusivity limits under mild conditions \cite{DL3,DL,MD3,PWZ,MD1,MD2,WXY,Winkler-large-1,Winkler-large-2}, as well as local stability of planar traveling waves \cite{mTW-1,mTW-2}.

Significant challenges remain, however, especially concerning non-constant steady states in bounded domains, precisely the setting of most experimental configurations. From an applied viewpoint, working on finite intervals introduces boundary effects that fundamentally alter solution behavior compared to the Cauchy problem. These boundary-induced phenomena are especially relevant for laboratory experiments in microfluidic chambers, tissue engineering scaffolds, and other confined geometries where boundary conditions actively influence pattern selection.

In a series of papers, we will construct explicit steady-state solutions for system \eqref{model} on bounded domains under Dirichlet boundary conditions, in both one and multiple spatial dimensions, accompanied by a rigorous analysis of their asymptotic stability. Owing to substantial technical challenges, primarily concerning the stability analysis, this comprehensive work must be divided into a sequence of individual studies. The present paper, as the first installment, addresses the foundational case of $\mathsf{m}=1$.

\subsection{Explicit Steady-State Solutions}

When $\mathsf{m}=1$, by taking the Cole-Hopf transformation: $v=(\ln c)_x$, system \eqref{model} is transformed into:
\begin{subequations}\label{Tmodel}
\begin{alignat}{2}
u_t &= \mathsf{d} u_{xx} - \chi (uv)_x, \label{Tmodel-1} \\
v_t &= \varepsilon v_{xx} + 2\varepsilon v v_x - \mu u_x. \label{Tmodel-2}
\end{alignat}
\end{subequations}
The steady-state system we analyze then takes the form:
\begin{subequations}\label{ss}
\begin{alignat}{2}
\mathsf{d} \bar{u}_{xx} - \chi (\bar{u}\bar{v})_x &=0, \label{ss-1} \\
\varepsilon \bar{v}_{xx} + 2\varepsilon \bar{v} \bar{v}_x - \mu \bar{u}_x &=0,\label{ss-2}
\end{alignat}
\end{subequations}
Through extensive investigation, we have identified three distinct types of explicit solutions to system \eqref{ss}, summarized in Table~\ref{table-1}.
\begin{table}[h] 
\centering
\caption{Explicit solutions to system \eqref{ss}}\label{table-1}
\begin{tabular}{|c|c|c|}
\hline
Type & $\bar{u}$& $\bar{v}$\\
\hline
Power & $(\mathfrak{a}x+\mathfrak{b})^{-2}$ 
& $-2\mathfrak{a}\mathsf{d}\chi^{-1}\,(\mathfrak{a}x+\mathfrak{b})^{-1}$ \\
\hline
Trigonometric & $\begin{aligned}
\sec^{2}(\mathfrak{a}x+\mathfrak{b})\\
\csc^{2}(\mathfrak{a}x+\mathfrak{b})
\end{aligned}$  
& $\begin{aligned}
2\mathfrak{a}\mathsf{d}\chi^{-1}\,\tan(\mathfrak{a}x+\mathfrak{b})\\
-2\mathfrak{a}\mathsf{d}\chi^{-1}\,\cot(\mathfrak{a}x+\mathfrak{b})
\end{aligned}$ \\
\hline
Hyperbolic & $\operatorname{csch}^{2}(\mathfrak{a}x+\mathfrak{b})$ 
& $-2\mathfrak{a}\mathsf{d}\chi^{-1}\,\operatorname{coth}(\mathfrak{a}x+\mathfrak{b})$ \\
\hline
\end{tabular}
\end{table}
We stress that various conditions on $\mathfrak{a}$, $\mathfrak{b}$, and boundary data are required to ensure validity of the solutions, a detailed discussion is deferred to the next section. Notably, while the power function case recovers known steady states on the half-line \cite{FZZ}, the trigonometric and hyperbolic solutions are novel contributions. The singular nature of secant and cosecant functions makes them exclusive to bounded domains, underscoring how spatial confinement generates solution behaviors not observed in infinite domains. These explicit solutions offer valuable reference cases for testing computational algorithms, particularly for verifying code performance near singularities and boundaries.

A major application of our analytical results lies in providing exact benchmarks for numerical validation. The explicit solutions allow precise assessment of computational schemes for handling the challenging logarithmic nonlinearity and associated boundary conditions. Furthermore, our stability analysis identifies parameter regions where patterns persist or dissolve, information crucial for designing experimental systems with targeted dynamical properties. Specifically, we study the stability of system \eqref{model} under the parameter regime:
\begin{align}\label{para}
\mathsf{d}>0,\qquad \varepsilon>0,\qquad \chi\mu>0.
\end{align}
While positivity of the diffusion coefficients ensures nontrivial steady-state solutions, the condition $\chi\mu>0$ corresponds to biologically stable scenarios: either chemoattraction with chemical consumption or chemorepulsion with chemical production. This condition also ensures well-posedness via hyperbolic stabilization in the associated balance law formulation. Our approach illustrates how carefully constructed special solutions can elucidate general model behavior, especially regarding the interplay between domain size, boundary conditions, and pattern formation.

The rest of the paper is organized as follows: Section~2 details the construction and validity conditions for explicit steady states. Section~3 establishes nonlinear stability results with explicit parameter thresholds. Section~4 presents numerical verification and explores dynamics beyond analytically tractable regimes. Section~5 closes the paper with concluding remarks and future research directions.


\section{Validity conditions} 

This section delineates the conditions under which the explicit steady-state solutions summarized in Table~\ref{table-1} are mathematically valid. We consider system \eqref{Tmodel} defined on the bounded interval $\mathcal{I} = (0,1)$, subject to Dirichlet-type boundary conditions:
\begin{align}\label{BC0}
u(0) = \alpha_1, \quad u(1) = \alpha_2; \qquad v(0) = \beta_1, \quad v(1) = \beta_2.
\end{align}
To construct explicit solutions of system \eqref{ss} satisfying \eqref{BC0}, a natural approach is to integrate \eqref{ss-1} once (ignoring constant of integration), leading to the ansatz:
$$
\bar{v}=\frac{\mathsf{d}}{\chi} \frac{\bar{u}_x}{\bar{u}}.
$$
Substituting this expression into \eqref{ss-2} yields a third-order ordinary differential equation for $\bar{u}$:
\begin{align}\label{ss-3}
\varepsilon \mathsf{d} \chi \bar{u}_{xxx}\bar{u}^2 - \varepsilon\mathsf{d}(3\chi-2\mathsf{d})\bar{u}_{xx}\bar{u}_x\bar{u} + 2\varepsilon\mathsf{d}(\chi-\mathsf{d})\bar{u}_x^3 - \mu \chi^2\bar{u}_x\bar{u}^3 = 0.
\end{align}
Through direct calculation, it is readily verified that the functions presented in Table \ref{table-1} solve \eqref{ss-3} and \eqref{BC0}, subject to the parametric and boundary constraints elaborated in the following subsection.

\subsection{Constraints and Monotonicity}
We focus on the case of $\chi>0$ \& $\mu>0$, which is consistent with the last requirement specified in \eqref{para}. The other case, i.e., $\chi<0$ \& $\mu<0$,  can be treated in a similar fashion. For fixed parameters $\mathsf{d}>0$, $\varepsilon>0$, $\chi>0$, and $\mu>0$, a direct calculation shows that for all the four types of functions specified in Table \ref{table-1}, the parameter $\mathfrak{a}$ is given by 
\begin{align}
\mathfrak{a}=\pm \kappa \triangleq \pm\sqrt{\frac{\mu\chi^2}{2\mathsf{d}\varepsilon(2\mathsf{d}+\chi)}}.
\end{align}
Next, we determine the value of $\mathfrak{b}$ and specify the constraints on the boundary values. 

First of all, since $u$ represents cell density and we are interested in the existence of non-constant steady-state solutions, it is natural to require 
\begin{align}
0<\alpha_1=u(0) \neq u(1)=\alpha_2>0.
\end{align}
Second, because all steady-state solutions involve the squares of elementary functions, $\mathfrak{b}$ can take either positive or negative values. Combined with the alternating sign of $\mathfrak{a}$, a full characterization of the parametric and boundary constraints becomes rather involved. To keep the presentation accessible, we restrict our discussion to the case where both \underline{$\mathfrak{a}$ and $\mathfrak{b}$ are positive}. The remaining cases can be analyzed in a similar manner. Furthermore, we will follow the same convention in the next section when analyzing the asymptotic stability of the steady-state solutions.

For the {\bf power function}, letting $x=0$ in $(\mathfrak{a}x+\mathfrak{b})^{-2}$, we obtain
\begin{align}
\mathfrak{b}=\alpha_1^{-1/2}.
\end{align}
Using $\mathfrak{a}$ and $\mathfrak{b}$, letting $x=1$ in $(\mathfrak{a}x+\mathfrak{b})^{-2}$, the right-boundary value of $u$ satisfies the constraint:
\begin{equation}\label{R2}
\begin{aligned}
\alpha_2^{-1/2}=\mathfrak{a}+\mathfrak{b}=\kappa+\alpha_1^{-1/2}.
\end{aligned}
\end{equation}
In the same way, $\beta_1$ and $\beta_2$ are determined by (see Table \ref{table-1})
\begin{align}
\beta_1=-\frac{2\kappa \mathsf{d}}{\mathfrak{b}\chi},\qquad 
\frac{1}{\beta_2} = \frac{1}{\beta_1} -\frac{\chi}{2\mathsf{d}}.
\end{align}
The constraints for the trigonometric and hyperbolic functions are listed below:

\vspace{.1 in}

$\bullet$\ \ $\boldsymbol{\sec^2}$ \& $\boldsymbol{\tan}$:
\begin{equation}
\begin{aligned}
\mathfrak{b}=\sec^{-1}(\sqrt{\alpha_1}); \quad &\alpha_1\geqslant 1, \ \kappa + \sec^{-1}(\sqrt{\alpha_1}) < \frac{\pi}{2}; \quad \sec^{-1}(\sqrt{\alpha_2})=\sec^{-1}(\sqrt{\alpha_1})+\kappa; \\ 
&\beta_1=2\kappa \mathsf{d}\chi^{-1}\tan{\mathfrak{b}},\quad \tan^{-1}\left(\frac{\beta_2\chi}{2\kappa \mathsf{d}}\right)=\tan^{-1}\left(\frac{\beta_1\chi}{2\kappa\mathsf{d}}\right) + \kappa.
\end{aligned}
\end{equation}

$\bullet$\ \ $\boldsymbol{\csc^2}$ \& $\boldsymbol{\cot}$: 
\begin{equation}
\begin{aligned}
\mathfrak{b}=\csc^{-1}(\sqrt{\alpha_1}); \quad &\alpha_1\geqslant 1, \ \kappa + \csc^{-1}(\sqrt{\alpha_1}) < \frac{\pi}{2}; \quad \csc^{-1}(\sqrt{\alpha_2})=\csc^{-1}(\sqrt{\alpha_1})+\kappa; \\
&\beta_1=-2\kappa\mathsf{d}\chi^{-1}\cot{\mathfrak{b}},\quad \cot^{-1}\left(-\frac{\beta_2\chi}{2\kappa\mathsf{d}}\right)=\cot^{-1}\left(-\frac{\beta_1\chi}{2\kappa \mathsf{d}}\right) + \kappa.
\end{aligned}
\end{equation}

$\bullet$\ \ $\boldsymbol{\operatorname{csch}^2}$ \& $\boldsymbol{\coth}$:
\begin{equation}
\begin{aligned}
\mathfrak{b}=\operatorname{csch}^{-1}\sqrt{\alpha_1}; \quad 
&\operatorname{csch}^{-1}(\sqrt{\alpha_2})=\operatorname{csch}^{-1}(\sqrt{\alpha_1})+\kappa; \\
&\beta_1=-2\kappa \mathsf{d}\chi^{-1}\operatorname{coth}{\mathfrak{b}},\quad \operatorname{coth}^{-1}\left(-\frac{\beta_2\chi}{2\kappa \mathsf{d}}\right) = \operatorname{coth}^{-1}\left(-\frac{\beta_1\chi}{2\kappa \mathsf{d}}\right)+\kappa.
\end{aligned}
\end{equation}
It should be mentioned that, unlike the power and hyperbolic functions, when deriving the constraints for the secant and cosecant functions, we restrict their domain to be within the interval $(0,\frac{\pi}{2})$, i.e., we require that $\mathfrak{a}x+\mathfrak{b} \in (0,\frac{\pi}{2})$ for $\forall\,x\in[0,1]$. Under such a restriction, both functions are continuous, bounded, and more importantly, strictly monotonic. The readers will see that the monotonicity of the steady-state solutions plays a pivotal role for establishing their nonlinear stability. Following direct calculations, the monotonicity of the functions in Table \ref{table-1} can be determined according to the sign of $\mathfrak{a}$. The results are summarized in Table \ref{table-2}.
\begin{table}[h]
\centering
\caption{Monotonicity of functions in Table \ref{table-1}}\label{table-2}
\begin{tabular}{|c|c|c|}
\hline
Type & $\bar u$ & $\bar v$ \\
\hline
power & $\searrow$ & $\nearrow$ \\			
\hline
$\sec^2$ \& $\tan$ & $\nearrow$ & $\nearrow$ \\
\hline
$\csc^2$ \& $\cot$ & $\searrow$ & $\nearrow$ \\
\hline
$\operatorname{csch}^2$ \& $\coth$ & $\searrow$ & $\nearrow$ \\
\hline
\end{tabular}
\end{table}

\subsection{Exclusivity}
For fixed parameter values and given boundary conditions, a natural question is whether the steady-state solution is unique or not. Though this question can be answered by using the usual approach based on energy method, here we present a more fundamental argument based on function properties. For this purpose, consider for instance the steady states given by the cosecant and hyperbolic functions, i.e., $\csc^2(\mathfrak{a}x+\mathfrak{b})$ and $\operatorname{csch}^2(\mathfrak{a}x+\mathfrak{c})$, for the same $\mathfrak{a}$. 
Suppose that the two functions satisfy the same boundary conditions, i.e.,
\begin{align*}
\csc^2\mathfrak{b}&=\alpha_1, \hspace{.35 in} \csc^2(\mathfrak{a}+\mathfrak{b})=\alpha_2; \\
\operatorname{csch}^2\mathfrak{c}&=\alpha_1, \qquad \operatorname{csch}^2(\mathfrak{a}+\mathfrak{c})=\alpha_2.
\end{align*}
Since $\mathfrak{b}>0$, it suffices to consider the following scenario:  
\begin{align*}
\mathfrak{b}&=\csc^{-1}(\sqrt{\alpha_1}), \hspace{.36 in} \mathfrak{a}+\mathfrak{b}=\csc^{-1}(\sqrt{\alpha_2});\\
\mathfrak{c}&=\operatorname{csch}^{-1}(\sqrt{\alpha_1}), \qquad 
\mathfrak{a}+\mathfrak{c}=\operatorname{csch}^{-1}(\sqrt{\alpha_2}),
\end{align*}
which implies the following equation:
\begin{align}\label{x1}
\csc^{-1}(\sqrt{\alpha_1})-\operatorname{csch}^{-1}(\sqrt{\alpha_1})=\csc^{-1}(\sqrt{\alpha_2})-\operatorname{csch}^{-1}(\sqrt{\alpha_2}).
\end{align}
Now, the question is whether the function $F(x)=\csc^{-1}(\sqrt{x})-\operatorname{csch}^{-1}(\sqrt{x})$ is one-to-one or not. 
Differentiating $F(x)$ with respect to $x$, we have
\begin{align*}
F^\prime(x)=-\frac{1}{2x\sqrt{{x}-1}}+\frac{1}{2x\sqrt{{x}+1}}.
\end{align*}
Obviously, $F^\prime$ is strictly negative for $x\geqslant 1$. Hence, equation \eqref{x1} holds only when $\alpha_1=\alpha_2$, which, however, is out of the consideration of this paper. Therefore, the two functions do not coexist. The proof for the mutual exclusion between other functions is similar, and we omit it for brevity.

The mutual exclusivity of steady states can also be understood by examining the original system \eqref{model}:
\begin{align*}
u_t &= \mathsf{d} u_{xx} - \chi [u(\ln c)_x]_x, \\
c_t &= \varepsilon c_{xx} - \sigma c - \mu u c.
\end{align*}
It is straightforward to verify that the power-type steady state satisfies the system only when $\sigma = 0$; the two trigonometric-type steady states are valid only when $\sigma < 0$; and the hyperbolic-type steady state applies only when $\sigma > 0$. Moreover, the two trigonometric-type steady states exhibit opposite monotonicity. Consequently, for a given set of parameters and boundary conditions, the four types of steady states are mutually exclusive.


\section{Asymptotic Stability}

In this section, we investigate the nonlinear asymptotic stability of the steady-state solutions constructed in the preceding section. For this purpose, we consider the initial-boundary value problem for system \eqref{Tmodel} subject to the boundary conditions specified in \eqref{BC0} and initial conditions:
\begin{align*} 
(u,v)(x,0)=(u_0,v_0)(x),\quad x \in [0,1].
\end{align*}
Throughout this section, we focus on the case in which $\chi>0$, $\mu>0$, $\mathfrak{a}>0$, and $\mathfrak{b}>0$, where $\mathfrak{a}$ and $\mathfrak{b}$ are the constants associated with the steady-state solutions discussed in the preceding section. The discussions for other cases are similar, and we leave the details for interested readers.

In contrast to the analysis for trivial steady states, a major challenge in the non-trivial case arises from the non-vanishing derivatives of the steady states. These introduce additional terms in the energy estimates, which require more careful treatment and stronger assumptions to establish nonlinear stability. We now impose further parametric and boundary constraints to ensure the negativity of the quantity:
\begin{align}\label{x2}
\frac{\bar{u}_{xx}}{2} - 2\bar{v}\bar{u}_x + (\bar{u}\bar{v})_x.
\end{align} 
This is necessary because, in our stability analysis, this quantity appears alongside $\tilde{v}^2$ on the right-hand side of the energy inequality. If the quantity is negative, the corresponding integral term can be discarded, thereby simplifying the analysis. Otherwise, one would need to impose smallness conditions on the derivatives of the steady-state solution, which is an undesirable requirement.

\begin{proposition}\label{prop}
Let $(\bar{u},\bar{v})$ be any of the steady-state solution constructed in Section 2 and satisfy the constraints specified therein. Let $\lambda \triangleq \frac{4\mathsf{d}-2\chi}{2\mathsf{d}-3\chi}$. Then the quantity specified in \eqref{x2} is non-negative, provided that the following conditions are satisfied:
\begin{table}[htbp]
\centering
\begin{tabular}{|c|c|}
\hline
{\rm power} & $3\chi\leqslant 2\mathsf{d}$ \\ 
\hline
$\sec^2$ \& $\tan$ & $3\chi < 2\mathsf{d}$ \ \& \ $\alpha_1 \geqslant \lambda$ \\ 
\hline
$\csc^2$ \& $\cot$ & $3\chi < 2\mathsf{d}$ \ \& \ $\alpha_2 \geqslant \lambda$ \\
\hline
$\operatorname{csch}^2$ \& $\coth$ & $3\chi \leqslant 2\mathsf{d}$ \quad {\rm or} \quad $6\mathsf{d} > 3\chi > 2\mathsf{d}$ \ \& \ $\alpha_1 \leqslant -\lambda$ \\ 
\hline
\end{tabular}
\end{table}
\end{proposition}

\begin{proof}
We present only the proof for the hyperbolic functions. The other cases follow a similar line of reasoning. To start off, we note that
\begin{align*}
\frac{\bar{u}_{xx}}{2} - 2\bar{v}\bar{u}_x + (\bar{u}\bar{v})_x = \frac{\chi + 2\mathsf{d}}{2\chi} \bar{u}_{xx} - \frac{2\mathsf{d}}{\chi} \frac{\bar{u}_x^2}{\bar{u}}.
\end{align*}
Since $\chi > 0$, it is equivalent to show 
\begin{align}\label{x3}
\bar u_{xx} - \frac{4\mathsf{d}}{\chi+2\mathsf{d}} \frac{\bar{u}_x^2}{\bar{u}} \leqslant 0.
\end{align}
Substituting the relevant expressions of steady-state solution into \eqref{x3}, we obtain
\begin{align*}
\operatorname{csch}^2(\mathfrak{a}x+\mathfrak{b}) \leqslant \frac{4\mathsf{d} - 2\chi}{\chi+2\mathsf{d}}\,\coth^2(\mathfrak{a}x+\mathfrak{b}),
\end{align*}
which implies immediately $2\mathsf{d} > \chi$. Using $\coth^2x = \operatorname{csch}^2x + 1$, we deduce
\begin{align}\label{x4}
(3\chi-2\mathsf{d})\,\operatorname{csch}^2(\mathfrak{a}x+\mathfrak{b}) \leqslant (4\mathsf{d}-2\chi).
\end{align}
If $3\chi \leqslant 2\mathsf{d}$, then \eqref{x4} is obviously true. On the other hand, if $3\chi > 2\mathsf{d}$, we must have 
\begin{align}\label{x5}
\operatorname{csch}^2(\mathfrak{a}x+\mathfrak{b}) \leqslant \frac{4\mathsf{d}-2\chi}{3\chi-2\mathsf{d}} = -\lambda, \quad \forall\,x\in[0,1].
\end{align}
Since we have restricted $\mathfrak{a}x+\mathfrak{b}$ ($x\in [0,1]$) to the interval $(0,\frac{\pi}{2})$ and $\operatorname{csch}^2x$ is strictly decreasing on such an interval, inequality \eqref{x5} is equivalent to 
\begin{align}\label{x6}
\operatorname{csch}^2(\mathfrak{b}) \leqslant -\lambda.
\end{align}
From Section 2, we know that $\operatorname{csch}^2(\mathfrak{b})=\alpha_1$. Therefore, combining the above results, we see that inequality \eqref{x6} is fulfilled if $3\chi \leqslant 2\mathsf{d}$ or $6\mathsf{d} > 3\chi > 2\mathsf{d}$ \& $\alpha_1 \leqslant -\lambda$. This completes the proof. 
\end{proof}

Our theorem is stated for the perturbed problem. For this, subtracting \eqref{ss} from \eqref{Tmodel}, we obtain
\begin{subequations}\label{PS}
\begin{align}
\tilde{u}_t &= \mathsf{d}\tilde{u}_{xx} - \chi ( \tilde{u}\tilde{v} + \bar{u}\tilde{v} + \bar{v}\tilde{u} )_x, \label{PS-1} \\
\tilde{v}_t &= \varepsilon\tilde{v}_{xx} + \varepsilon ( \tilde{v}^2 + 2\bar{v}\tilde{v} )_x - \mu \tilde{u}_x, \label{PS-2}
\end{align}
\end{subequations}
where $\tilde u = u - \bar u$ and $\tilde v = v - \bar v$. The initial and boundary conditions for system \eqref{PS} read as:
\begin{align}
\tilde{u}(x,0) &= (u_0-\bar{u})(x), \quad \tilde{v}(x,0) = (v_0-\bar{v})(x), \qquad x\in[0,1], \label{IC}\\
\tilde{u}(x,t)&=0,\hspace{.85 in} \tilde{v}(x,t)=0, \hspace{.95 in} x=0,1,\quad t \geqslant 0. \label{BC}
\end{align}
We now present the stability result, following the introduction of some notations for brevity.

\begin{notation} 
Throughout this section, unless otherwise specified, we denote by $C$ a generic positive constant that is independent of the unknown functions and time. We use $\| \cdot \|$, $\| \cdot \|_{\infty}$, and $\|\cdot\|_k$ to represent the norms $\| \cdot \|_{L^2(\mathcal{I})}$, $\| \cdot \|_{L^{\infty}(\mathcal{I})}$, and $\|\cdot\|_{H^k(\mathcal{I})}$, respectively. Furthermore, for a suitably chosen weight function $w$, we introduce a weighted norm defined by $\|f\|^2_{k,\omega}= \sum _{j=0}^k\|\sqrt{\omega}\,\partial_x^jf\|^2$.
\end{notation}

\begin{theorem}\label{thm}
Consider the initial-boundary value problem \eqref{PS}--\eqref{BC} with positive parameters. Let $(\bar{u},\bar{v})$ be any of the steady-state solutions listed in Table \ref{table-1} with positive constants $\mathfrak{a}$ and $\mathfrak{b}$, and let the parameters and boundary values satisfy the corresponding constraints specified in Section 2 and Proposition \ref{prop}. Let the initial functions $\tilde{u}_0, \tilde{v}_0 \in H^2(0,1)$ be compatible with the boundary conditions. Suppose that  there exists a sufficiently small constant $\delta_0>0$ such that $\|\tilde{u}_0\|^2+\|\tilde{v}_0\|^2\leqslant \delta_0$. Then there exists a unique solution to the IBVP \eqref{PS}--\eqref{BC} such that $\tilde u, \tilde v \in L^{\infty}((0,\infty);H^2(0,1)) \cap L^{2}((0,\infty);H^3(0,1))$
and $\|\tilde{u}\|_2^2 + \|\tilde{v}\|_2^2$ converges to zero exponentially rapidly as $t \to \infty$.
\end{theorem}

\begin{remark}
In the present work, the smallness assumption is imposed only on the $L^2$ norm of the initial perturbation, while $\|(\tilde{u}_{0x}, \tilde{v}_{0x})\|_1$ is allowed to be large. In this regard, we establish the nonlinear stability of the steady-state solutions under so-called ``partially large perturbations". Moreover, Theorem \ref{thm} does not impose any smallness condition on the steady-state solutions. 
\end{remark}

\begin{remark}
It is worth mentioning that when $\chi<0$, stronger assumptions on the derivatives of $\bar{u}$ and $\bar{v}$ would be required to establish the stability result. The proof for this case is indeed simpler than the one for the case when $\chi>0$.
\end{remark}

As usual, the proof of Theorem \ref{thm} proceeds in three steps. First, we establish short-time well-posedness of the IBVP, which can be shown via the Galerkin and fixed-point methods. Second, we derive uniform-in-time {\it a priori} estimates of the local. Finally, we extend the local solution to global solution by a standard continuation argument. We first present the local well-posedness result.

\begin{lemma}\label{lem}
Under the hypotheses of Theorem \ref{thm}, there exists a finite $T>0$ and a unique solution to the IBVP \eqref{PS}--\eqref{BC}, such that $\tilde u,\tilde{v} \in L^{\infty}\left((0,T);H^2(0,1)\right) \cap L^{2}((0,T);H^3(0,1))$.
\end{lemma}

The bulk of this section is devoted to deriving the uniform-in-time {\it a priori} estimates of the local solution obtained in Lemma \ref{lem}. Hereafter we assume that $\|\tilde{u}(t)\|+\|\tilde{v}(t)\|$ remains uniformly small throughout the lifespan of the local solution. Under this hypothesis, the {\it a priori} bounds in the relevant Sobolev spaces are shown to be uniform in time. All subsequent energy estimates are derived under the following {\it a priori} assumption:
\begin{align}\label{y1}
\sup_{t\in [0,T]} \left(\|\tilde{u}(t)\|+\|\tilde{v}(t)\|\right) \leqslant \delta_1,\qquad \sup_{t\in [0,T]} \left(\|\tilde{u}_x(t)\|+\|\tilde{v}_x(t)\|\right) \leqslant \mathsf{M},
\end{align}
where $\delta_1>0$ is a small constant, $\mathsf{M}>0$ is a finite constant, and $T>0$ denotes any finite time within the lifespan of the local solution. The proof of the {\it a priori} estimates is structured into five steps, the first four deal with the regularity and uniform-in-time estimates, while the last one addresses the decay of the perturbation. We begin with the zeroth-order estimate of the perturbation.

\begin{lemma}\label{lem1} 
Under the hypotheses of Theorem \ref{thm}, there exists a constant $C>0$, such that
$$
\|\tilde{u}(t)\|^2+\|\tilde{v}(t)\|_{\omega}^2+\int_0^t \left(\|\tilde{u}_x(\tau)\|^2+\|\tilde{v}_x(\tau)\|_{\omega}^2\right)\mathrm{d}\tau \leqslant C\left(\|\tilde{u}_0\|^2+\|\tilde{v}_0\|_{\omega}^2\right).
$$
\end{lemma}

\begin{proof} \textbf{Step 1.}
Taking $L^2$ inner product of \eqref{PS-1} with $\mu\tilde{u}$ and integrating by parts, we derive
\begin{align}\label{y2}
\frac{\mu}{2} \frac{\mathrm d}{\mathrm dt} \|\tilde{u}\|^2 + \mathsf{d}\mu \|\tilde{u}_x\|^2 = \chi\mu \int_0^1 \tilde{u} \tilde{u}_x \tilde{v} \, \mathrm{d}x + \chi\mu \int_0^1 \bar{u} \tilde{v} \tilde{u}_x \, \mathrm{d}x + \chi\mu \int_0^1 \bar{v} \tilde{u} \tilde{u}_x \, \mathrm{d}x.
\end{align}
For the first integral on the right-hand side of \eqref{y2}, by Cauchy's inequality, we have
\begin{align}\label{y3}
\bigg|\chi\mu \int_0^1 \tilde{u} \tilde{u}_x \tilde{v} \, \mathrm{d}x \bigg| \leqslant \chi\mu \|\tilde{u}\|_{\infty} \int_0^1 |\tilde{u}_x \tilde{v}| \, \mathrm{d}x \leqslant \frac{\chi\mu \|\tilde{u}\|_{\infty}}{2} \left( \|\tilde{u}_x\|^2 + \|\tilde{v}\|^2 \right).
\end{align}
Since $\tilde{u}$ vanishes on the boundary, by interpolation and \eqref{y1}, we obtain
\begin{align}\label{y4}
\|\tilde{u}\|_{\infty} \leqslant \sqrt{2} \|\tilde{u}\|^{1/2} \|\tilde{u}_x\|^{1/2} \leqslant \sqrt{2\delta_1 \mathsf{M}}.
\end{align}
Let $\omega=\chi\bar u$. From the monotonicity of the steady-state solution $\bar{u}$ and positivity of the boundary values, we know that there exist constants $\overline{\omega}>\underline{\omega} >0$, such that the weight function $\omega$ satisfies
\begin{align}\label{y5}
\underline{\omega} \leqslant \omega \leqslant \overline{\omega}.
\end{align}
Combining \eqref{y3}, \eqref{y4}, \eqref{y5}, and Poincar\'{e}'s inequality leads to
\begin{align}\label{y6}
\bigg|\chi\mu \int_0^1 \tilde{u} \tilde{u}_x \tilde{v} \, \mathrm{d}x \bigg| \leqslant \frac{\chi\mu \sqrt{\delta_1 \mathsf{M}}}{\sqrt{2}} \left(\|\tilde{u}_x\|^2 + \frac{1}{\pi^2\underline{\omega}} \|\tilde{v}_x\|^2_{\omega} \right).
\end{align}
For the third integral on the right-hand side of \eqref{y2}, integration-by-parts gives
\begin{align}\label{y7}
\chi\mu \int_0^1 \bar{v} \tilde{u} \tilde{u}_x \, \mathrm{d}x = -\frac{\chi\mu}{2} \int_0^1 \bar{v}_x \tilde{u}^2 \, \mathrm{d}x.
\end{align}
Since $\chi\mu > 0$, and $\bar{v}$ is increasing (c.f. Table \ref{table-2}), it follows that the above integral is negative. Combining \eqref{y2}, \eqref{y6}, and \eqref{y7}, we obtain
\begin{align}\label{y8}
\frac{\mu}{2} \frac{\mathrm d}{\mathrm dt} \|\tilde{u}\|^2 + \mathsf{d}\mu \|\tilde{u}_x\|^2 \leqslant \frac{\chi\mu \sqrt{\delta_1 \mathsf{M}}}{\sqrt{2}} \left(\|\tilde{u}_x\|^2 + \frac{1}{\pi^2\underline{\omega}} \|\tilde{v}_x\|^2_{\omega} \right) + \chi\mu \int_0^1 \bar{u} \tilde{v} \tilde{u}_x \, \mathrm{d}x.
\end{align}

\textbf{Step 2.}
Taking $L^2$ inner product of \eqref{PS-2} with $\chi\bar{u} \tilde{v}$ and integrating by parts yield
\begin{align}\label{y9}
\frac{1}{2} \frac{\mathrm d}{\mathrm dt}\|\tilde{v}\|_{\omega}^2 + \varepsilon\|\tilde{v}_x\|_{\omega}^2 = -\varepsilon \chi \int_0^1 \left[\tilde{v}_x \tilde{v} \bar{u}_x + (\tilde{v}^2 + 2\bar{v} \tilde{v}) (\bar{u} \tilde{v})_x\right] \mathrm{d}x - \chi \mu \int_0^1 \bar{u} \tilde{u}_x \tilde{v} \, \mathrm{d}x.
\end{align}
Applying integration-by-parts, we rewrite the first integral on the right-hand side of \eqref{y9} as
\begin{align}\label{y10}
 -\frac{2\varepsilon\chi}{3} \int_0^1 \bar{u}_x \tilde{v}^3\, \mathrm{d}x + \varepsilon\chi\int_0^1 \left[ \frac{\bar{u}_{xx}}{2} - 2\bar{v} \bar{u}_x + (\bar{u}\bar{v})_x\right] \tilde{v}^2 \, \mathrm{d}x.
\end{align}
By the construction of the steady state, we know that there exist constants $\overline{\mathsf{u}} > \underline{\mathsf{u}}>0$, such that 
\begin{align}\label{y11}
\underline{\mathsf{u}} \leqslant \|\bar{u}_x\|_\infty \leqslant \overline{\mathsf{u}}.
\end{align}
Since $\tilde{v}$ vanishes on the boundary, similar to \eqref{y4}, we have
\begin{align}\label{y12}
\|\tilde{v}\|_{\infty} \leqslant \sqrt{2} \|\tilde{v}\|^{1/2} \|\tilde{v}_x\|^{1/2} \leqslant \sqrt{2\delta_1 \mathsf{M}}.
\end{align}
Combining \eqref{y11}, \eqref{y12}, and applying Poincar\'e's inequality, we estimate the first integral on the right-hand side of \eqref{y10} as 
\begin{align}\label{y13}
\bigg|\frac{2\varepsilon\chi}{3} \int_0^1 \bar{u}_x \tilde{v}^3 \, \mathrm{d}x\bigg|  
\leqslant \frac{2\varepsilon \chi \overline{\mathsf{u}}}{3} \|\tilde{v}\|_{\infty} \|\tilde{v}\|^2 \leqslant \frac{2\varepsilon \chi \overline{\mathsf{u}} \sqrt{2\delta_1\mathsf{M}}}{3\pi^2 \underline{\omega}} \|\tilde{v}_x\|_{\omega}^2.
\end{align}
The combination of \eqref{y9}, \eqref{y10}, and \eqref{y13} gives 
\begin{align}\label{y14}
\frac{1}{2} \frac{\mathrm d}{\mathrm dt}\|\tilde{v}\|_{\omega}^2 + \varepsilon\|\tilde{v}_x\|_{\omega}^2 \leqslant \frac{2\varepsilon \chi \overline{\mathsf{u}} \sqrt{2\delta_1\mathsf{M}}}{3\pi^2 \underline{\omega}} \|\tilde{v}_x\|_{\omega}^2 + \varepsilon\chi\int_0^1 \mathfrak{L}(x) \tilde{v}^2 \, \mathrm{d}x - \chi \mu \int_0^1 \bar{u} \tilde{u}_x \tilde{v} \, \mathrm{d}x,
\end{align}
where the function $\mathfrak{L}$ is defined by 
\begin{align*}
\mathfrak{L}(x) = \frac{\bar{u}_{xx}}{2} - 2\bar{v} \bar{u}_x + (\bar{u}\bar{v})_x,\quad x\in [0,1].
\end{align*}
From Proposition \ref{prop}, we know that $\mathfrak{L} \leqslant 0$. Hence, we update \eqref{y14} as 
\begin{align}\label{y15}
\frac{1}{2} \frac{\mathrm d}{\mathrm dt}\|\tilde{v}\|_{\omega}^2 + \varepsilon\|\tilde{v}_x\|_{\omega}^2 \leqslant \frac{2\varepsilon \chi \overline{\mathsf{u}} \sqrt{2\delta_1\mathsf{M}}}{3\pi^2 \underline{\omega}} \|\tilde{v}_x\|_{\omega}^2 - \chi \mu \int_0^1 \bar{u} \tilde{u}_x \tilde{v} \, \mathrm{d}x.
\end{align}

\textbf{Step 3.}
Adding \eqref{y8} and \eqref{y15} together, we obtain
\begin{align*}
\frac{\mathrm d}{\mathrm dt} \left( \frac{\mu}{2} \|\tilde{u}\|^2 + \frac{1}{2} \|\tilde{v}\|_{\omega}^2 \right) + \mathsf{d}\mu \|\tilde{u}_x\|^2 + \varepsilon \|\tilde{v}_x\|_{\omega}^2 \leqslant \mathsf{c_1}\|\tilde{u}_x\|^2 + \mathsf{c}_2 \|\tilde{v}_x\|^2, 
\end{align*}
where the constants $\mathsf{c}_1$ and $\mathsf{c}_2$ are defined by 
\begin{align*}
\mathsf{c}_1 = \frac{\chi\mu \sqrt{\delta_1 \mathsf{M}}}{\sqrt{2}},\qquad \mathsf{c}_2 = \frac{\chi\mu \sqrt{\delta_1 \mathsf{M}}}{\sqrt{2}\pi^2\underline{\omega}} + \frac{2\varepsilon \chi \overline{\mathsf{u}} \sqrt{2\delta_1\mathsf{M}}}{3\pi^2 \underline{\omega}}.
\end{align*}
When $\delta_1$ is sufficiently small such that $\mathsf{c}_1 \leqslant \frac{\mathsf{d}\mu}{2}$ and $\mathsf{c}_2 \leqslant \frac{\varepsilon}{2}$, we have 
\begin{align}\label{y16}
\frac{\mathrm d}{\mathrm dt} \left(\mu \|\tilde{u}\|^2 + \|\tilde{v}\|_{\omega}^2 \right) + \mathsf{d}\mu \|\tilde{u}_x\|^2 + \varepsilon \|\tilde{v}_x\|_{\omega}^2 \leqslant 0.
\end{align}
Integrating \eqref{y16} with respect to $t$ completes the proof.
\end{proof}

\begin{lemma}\label{lem2}
Under the hypotheses of Theorem \ref{thm}, there exists a constant $C>0$, such that
$$
\|\tilde{u}_x(t)\|^2+\|\tilde{v}_x(t)\|^2+\int_0^t \left(\|\tilde{u}_{xx}(\tau)\|^2+\|\tilde{v}_{xx}(\tau)\|^2\right)\mathrm{d}\tau \leqslant C\left(\|\tilde{u}_0\|_{1}^2+\|\tilde{v}_0\|_1^2\right).
$$
\end{lemma}

\begin{proof} Multiplying \eqref{PS-1} by $-\tilde{u}_{xx}$ and integrating with respect to $x$, we obtain
\begin{align}\label{y17}
\frac{\mathrm d}{\mathrm dt} \|\tilde{u}_x\|^2 + \mathsf{d} \|\tilde{u}_{xx}\|^2 \leqslant \frac{\chi^2}{\mathsf{d}} \| (\tilde{u}\tilde{v} + \bar{u}\tilde{v} + \bar{v}\tilde{u})_x \|^2.
\end{align}
Since $\tilde{u}$ and $\tilde{v}$ satisfy Poincar\'e's inequality, by Sobolev embedding, we have
\begin{align*}
\|\tilde{u}\|_\infty^2 \leqslant C \|\tilde{u}_x\|^2,\qquad \|\tilde{v}\|_\infty^2 \leqslant C \|\tilde{v}_x\|^2,
\end{align*}
by which we derive 
\begin{align}\label{z1}
\|(\tilde{u}\tilde{v})_x\|^2 \leqslant C \|\tilde{u}_x\|^2\|\tilde{v}_x\|^2.
\end{align}
Using the smoothness of the steady-state solution, we can show that 
\begin{align}\label{z2}
\|(\bar{u}\tilde{v} + \bar{v}\tilde{u})_x \|^2 \leqslant C\left(\|\tilde{u}_x\|^2+\|\tilde{v}_x\|^2\right).
\end{align}
Then we update \eqref{y17} as 
\begin{align}\label{y18}
\frac{\mathrm d}{\mathrm dt} \|\tilde{u}_x\|^2 + \mathsf{d} \|\tilde{u}_{xx}\|^2 \leqslant C\left(\|\tilde{u}_x\|^2\|\tilde{v}_x\|^2+\|\tilde{u}_x\|^2+\|\tilde{v}_x\|^2\right).
\end{align}
Multiplying \eqref{PS-2} by $-\tilde{v}_{xx}$ and integrating with respect to $x$, we obtain
\begin{align}\label{y19}
\frac{\mathrm d}{\mathrm dt} \|\tilde{v}_x\|^2 + \varepsilon\|\tilde{v}_{xx}\|^2 
\leqslant  C\left(\|\tilde{u}_x\|^2 + \|\tilde{v}\tilde{v}_x\|^2 + \|(\bar{v}\tilde{v})_x\|^2\right).
\end{align}
Similarly, we have
\begin{align*}
\|\tilde{v}\tilde{v}_x\|^2 + \|(\bar{v}\tilde{v})_x\|^2 \leqslant C\left(\|\tilde{v}_x\|^2\|\tilde{v}_x\|^2 + \|\tilde{v}_x\|^2\right).
\end{align*} 
which updates \eqref{y19} as 
\begin{align}\label{y20}
\frac{\mathrm d}{\mathrm dt} \|\tilde{v}_x\|^2 + \varepsilon\|\tilde{v}_{xx}\|^2 
\leqslant  C\left(\|\tilde{v}_x\|^2\|\tilde{v}_x\|^2 + \|\tilde{u}_x\|^2 + \|\tilde{v}_x\|^2\right).
\end{align}
Taking the sum of \eqref{y18} and \eqref{y20}, we obtain
\begin{align}\label{z3}
\frac{\mathrm d}{\mathrm dt}\left(\|\tilde{u}_x\|^2 + \|\tilde{v}_x\|^2\right) + \mathsf{d}\|\tilde{u}_{xx}\|^2 + \varepsilon \|\tilde{v}_{xx}\|^2
\leqslant C\left( \|\tilde{u}_x\|^2\|\tilde{v}_x\|^2 + \|\tilde{v}_x\|^4  + \|\tilde{u}_x\|^2 + \|\tilde{v}_x\|^2 \right).
\end{align}
An application of Gr\"onwall's inequality, together with Lemma \ref{lem1}, completes the proof.
\end{proof}

As a by-product of Lemmas \ref{lem1} and \ref{lem2}, we have

\begin{lemma}\label{lem3}
Under the hypotheses of Theorem \ref{thm}, there exists a constant $C>0$, such that
$$
\int_0^t \left(\|\tilde u_t(\tau)\|^2 + \|\tilde v_t(\tau)\|^2\right) \mathrm{d}\tau \leqslant C.
$$
\end{lemma}

\begin{proof}
By \eqref{PS-1}, \eqref{z1}, \eqref{z2}, and Lemma \ref{lem2}, we can show that 
\begin{align*}
\|\tilde{u}_t\|^2 \leqslant C\left(\|\tilde{u}_{xx}\|^2 + \|\tilde{u}_x\|^2 + \|\tilde{v}_x\|^2\right).
\end{align*}
Similarly, by \eqref{PS-2}, we have 
\begin{align*}
\|\tilde{v}_t\|^2 \leqslant C\left(\|\tilde{v}_{xx}\|^2 + \|\tilde{u}_x\|^2 + \|\tilde{v}_x\|^2\right).
\end{align*}
Integrating with respect to $t$ and applying Lemmas \ref{lem1} and \ref{lem2} complete the proof.
\end{proof}

\begin{lemma}\label{lem4}
Under the hypotheses of Theorem \ref{thm}, there exists a constant $C>0$, such that
$$
\|\tilde u_{xx}(t)\|^2+\|\tilde v_{xx}(t)\|^2+\int_0^t \left(\|\tilde u_{xxx}(\tau)\|^2 + \|\tilde v_{xxx}(\tau)\|^2\right) \mathrm{d}\tau \leqslant C.
$$
\end{lemma}

\begin{proof} Differentiating \eqref{PS} with respect to $t$, we obtain
\begin{subequations}
\begin{alignat}{2}
\tilde u_{tt} &= \mathsf{d} \tilde u_{xxt} - \chi (\tilde u \tilde v + \bar u \tilde v + \bar v \tilde u)_{xt}, \label{y21} \\
\tilde v_{tt} &= \varepsilon \tilde v_{xxt} - \mu \tilde u_{xt} + \varepsilon (\tilde v^2 + 2 \tilde v \bar v)_{xt}. \label{y22}
\end{alignat}
\end{subequations}
Multiplying \eqref{y21} by $\tilde u_t$ and integrating with respect to $x$, we have
\begin{align}\label{y23}
\frac{\mathrm{d}}{\mathrm{d}t} \|\tilde u_t\|^2 + \mathsf{d} \|\tilde u_{xt}\|^2 \leqslant C \|(\tilde u \tilde v + \bar u \tilde v + \bar v \tilde u)_t\|^2.
\end{align}
The right-hand side of \eqref{y23} can be estimated as 
\begin{align*}
\|(\tilde u \tilde v + \bar u \tilde v + \bar v \tilde u)_t\|^2 \leqslant C \left( \|\tilde v_x\|^2\|\tilde u_t\|^2 + \|\tilde u_x\|^2 \|\tilde v_t\|^2 + \|\tilde v_t\|^2 + \|\tilde u_t\|^2 \right).
\end{align*}
This updates \eqref{y23} as 
\begin{align}\label{y24}
\frac{\mathrm{d}}{\mathrm{d}t} \|\tilde u_t\|^2 + \mathsf{d} \|\tilde u_{xt}\|^2 \leqslant C \left( \|\tilde v_x\|^2\|\tilde u_t\|^2 + \|\tilde u_x\|^2 \|\tilde v_t\|^2 + \|\tilde v_t\|^2 + \|\tilde u_t\|^2 \right).
\end{align}
Similarly, we have 
\begin{align}\label{y25}
\frac{\mathrm{d}}{\mathrm{d}t} \|\tilde v_t\|^2 + \varepsilon \|\tilde v_{xt}\|^2 \leqslant C\left(\|\tilde{v}_x\|^2 \|\tilde{v}_t\|^2 + \|\tilde{u}_t\|^2 + \|\tilde{v}_t\|^2 \right).
\end{align}
Taking the sum of \eqref{y24} and \eqref{y25}, we obtain
\begin{align*}
\frac{\mathrm{d}}{\mathrm{d}t} \left( \|\tilde u_t\|^2 + \|\tilde v_t\|_{\omega}^2 \right) + \mathsf{d} \|\tilde u_{xt}\|^2 + \varepsilon \|\tilde v_{xt}\|_{\omega}^2 \leqslant C \left(\|\tilde{u}_x\|^2+\|\tilde v_x\|^2 + 1\right) \left( \|\tilde u_t\|^2 +\|\tilde v_t\|^2 \right).
\end{align*}
Applying Gr\"onwall's inequality, together with Lemmas \ref{lem1} and \ref{lem3},  we arrive at 
\begin{align}\label{y26}
\|\tilde u_t(t)\|^2 + \|\tilde v_t(t)\|_{\omega}^2 + \int_0^t\left(\|\tilde u_{xt}(\tau)\|^2 + \|\tilde v_{xt}(\tau)\|^2\right)\mathrm{d}\tau \leqslant C.
\end{align}
With \eqref{y26} at our disposal, the proof can be completed by repeatedly using the equations in \eqref{PS} and applying the established {\it a priori} estimates. We omit the details for brevity.
\end{proof}

The last lemma establishes the exponential decay of the perturbation.

\begin{lemma}\label{lem5}
Under the hypotheses of Theorem \ref{thm}, $\|\tilde u(t)\|_{2}^2+\|\tilde v(t)\|_2^2$ converges to zero exponentially rapidly as $t\to\infty$. 
\end{lemma}

\begin{proof}
The exponential decay of $\|\tilde{u}(t)\|^2+\|\tilde{v}(t)\|^2$ follows directly from \eqref{y16}, \eqref{y5}, and Poincar\'e's inequality. For the first order spatial derivative, by Lemma \ref{lem2}, we update \eqref{z3} as  
\begin{align}\label{y27}
\frac{\mathrm d}{\mathrm dt}\left(\|\tilde{u}_x\|^2 + \|\tilde{v}_x\|^2\right) + \mathsf{d}\|\tilde{u}_{xx}\|^2 + \varepsilon \|\tilde{v}_{xx}\|^2
\leqslant C\left(\|\tilde{u}_x\|^2 + \|\tilde{v}_x\|^2 \right).
\end{align}
By coupling \eqref{y16} and \eqref{y27} together, and applying Poincar\'e's inequality, we deduce 
\begin{align*}
\frac{\mathrm d}{\mathrm dt}\mathsf{X}(t) + \mathsf{Y}(t)\leqslant 0,
\end{align*}
where the quantities satisfy 
\begin{align}\label{y28}
\mathsf{X}(t) \cong \|\tilde{u}(t)\|_1^2 + \|\tilde{v}(t)\|_1^2,\qquad \mathsf{Y}(t) \cong \|\tilde{u}(t)\|_2^2 + \|\tilde{v}(t)\|_2^2.
\end{align}
The exponential decay of $\|\tilde{u}(t)\|_1^2 + \|\tilde{v}(t)\|_1^2$ follows directly from \eqref{y28}. The decay of the second order spatial derivative of the perturbation follows similarly, and we omit the details for brevity.
\end{proof}

\section{Numerical Experiment}

The explicit steady-state solutions presented in Table \ref{table-1} are derived under specific parametric and boundary constraints. In particular, the boundary conditions must satisfy the mutually exclusive restrictions outlined in Section 2, which are essential to ensure the validity of the solutions. These constraints correspond to distinct curves in the $\alpha_1\alpha_2$- and $\beta_1\beta_2$-planes, which do not span the entire parameter space. For points not located on these curves, our analytical approach does not yield any insight into the existence of nontrivial solutions to system \eqref{ss}, let alone their explicit forms. This analytical gap is not merely a theoretical shortcoming but a significant practical limitation, as real-world systems often operate under boundary conditions that do not conform to these idealizations. This underscores the critical need for complementary numerical investigations of the IBVP for system \eqref{Tmodel}, particularly in cases where the boundary conditions deviate from the constraints specified in Section 2.

To explore the behavior of the system beyond analytically tractable regimes and understand its potential responses in application-relevant scenarios, we conduct numerical simulations of the initial-boundary value problem on the unit interval $(0,1)$ using a finite difference scheme with mesh sizes $(\triangle x, \triangle t) = (0.002,0.001)$. The parameter values are chosen as follows:
\begin{align}\label{parameters}
\mathsf{d} = 2,\quad \chi = 1,\quad \varepsilon = 0.5,\quad \mu = 1.
\end{align}
We examine the following sets of boundary conditions:

\begin{table}[h]
\centering
\caption{Boundary conditions for numerical simulation of system \eqref{Tmodel}}\label{table-3}
\renewcommand{\arraystretch}{1.15}
\begin{tabular}{|c|c|c|c|c|}
\hline
 & {\scriptsize$u(0)$} & {\scriptsize$u(1)$} & {\scriptsize$v(0)$} & {\scriptsize$v(1)$} \\[1mm]
\hline
{\scriptsize power} & {\scriptsize$4$} & {\scriptsize$\frac{40}{(\sqrt{10}+2)^2}$} & {\scriptsize$-\frac{8}{\sqrt{10}}$} & {\scriptsize$-\frac{8}{\sqrt{10}+2}$} \\[1mm] 
\hline
{\scriptsize$\sec^2$ \& $\tan$} & {\scriptsize$4$} & {\scriptsize$\sec^{2}\left(\tfrac{1}{\sqrt{10}}+\tfrac{\pi}{3}\right)$} & {\scriptsize$\frac{4\sqrt{3}}{\sqrt{10}}$} & {\scriptsize$\frac{4}{\sqrt{10}}\tan\left(\tfrac{1}{\sqrt{10}}+\tfrac{\pi}{3}\right)$} \\[1mm] 
\hline
{\scriptsize$\csc^2$ \& $\cot$} & {\scriptsize$4$} & {\scriptsize$\csc^{2}\left(\sqrt{\tfrac{1}{10}}+\tfrac{\pi}{6}\right)$} & {\scriptsize$-\frac{4\sqrt{3}}{\sqrt{10}}$} & {\scriptsize$-\frac{4}{\sqrt{10}}\cot\left(\sqrt{\tfrac{1}{10}}+\tfrac{\pi}{6}\right)$} \\[1mm]
\hline
{\scriptsize$\operatorname{csch}^2$ \& $\coth$} & {\scriptsize$4$} & {\scriptsize$\operatorname{csch}^2\left(\sqrt{\tfrac{1}{10}} + \ln\!\tfrac{1+\sqrt{5}}{2}\right)$} & {\scriptsize$-\frac{4}{\sqrt{10}}$} & {\scriptsize$-\frac{4}{\sqrt{10}}\coth\left(\sqrt{\tfrac{1}{10}} + \ln\!\tfrac{1+\sqrt{5}}{2}\right)$} \\[1mm] 
\hline
{\scriptsize undiscovered} & {\scriptsize$4$} & {\scriptsize$21$} & {\scriptsize$-0.7$} & {\scriptsize$0.7$} \\[1mm]
\hline
\end{tabular}
\end{table}

\noindent It can be verified that, given the parameter values in \eqref{parameters}, the first four sets of boundary conditions in Table \ref{table-3} satisfy the constraints in Section 2, while the fifth one does not. The corresponding initial conditions in Table \ref{table-4} are chosen to be compatible with the boundary conditions.

\begin{table}[h]
\centering
\caption{Initial conditions for numerical simulation of  system \eqref{Tmodel}}\label{table-4}
\renewcommand{\arraystretch}{1.15}
\begin{tabular}{|c|c|c|}
\hline
 & {\scriptsize$u_0$} & {\scriptsize$v_0$} \\[1mm]
\hline
{\scriptsize power} & {\scriptsize$4+\left(\frac{40}{(\sqrt{10}+2)^2}-4\right)x$} & {\scriptsize$-\frac{8}{\sqrt{10}} + \left(\frac{8}{\sqrt{10}}-\frac{8}{\sqrt{10}+2}\right)x$} \\[1mm] 
\hline
{\scriptsize$\sec^2$ \& $\tan$} & {\scriptsize$4+\left(\sec^{2}\left(\tfrac{1}{\sqrt{10}}+\tfrac{\pi}{3}\right)-4\right)x$} & {\scriptsize$\frac{4\sqrt{3}}{\sqrt{10}} + \left(\frac{4}{\sqrt{10}}\tan\left(\tfrac{1}{\sqrt{10}}+\tfrac{\pi}{3}\right) - \frac{4\sqrt{3}}{\sqrt{10}}\right)x$} \\[1mm] 
\hline
{\scriptsize$\csc^2$ \& $\cot$} & {\scriptsize$4+\left(\csc^{2}\left(\sqrt{\tfrac{1}{10}}+\tfrac{\pi}{6}\right)-4\right)x$} & {\scriptsize$-\frac{4\sqrt{3}}{\sqrt{10}} + \left(-\frac{4}{\sqrt{10}}\cot\left(\sqrt{\tfrac{1}{10}}+\tfrac{\pi}{6}\right) + \frac{4\sqrt{3}}{\sqrt{10}}\right)x$} \\[1mm]
\hline
{\scriptsize$\operatorname{csch}^2$ \& $\coth$} & {\scriptsize$4+(\operatorname{csch}^2\left(\sqrt{\tfrac{1}{10}} + \ln\!\tfrac{1+\sqrt{5}}{2}\right)-4)\,x$} & {\scriptsize$-\frac{4}{\sqrt{10}} + \left(-\frac{4}{\sqrt{10}}\coth\left(\sqrt{\tfrac{1}{10}} + \ln\tfrac{1+\sqrt{5}}{2}\right)+\frac{4}{\sqrt{10}}\right)x$} \\[1mm] 
\hline
{\scriptsize undiscovered} & {\scriptsize$4+17\sin\left(\tfrac{\pi}{2}x\right)$} & {\scriptsize$-0.7+1.4\sin\left(\tfrac{\pi}{2}x\right)$} \\[1mm]
\hline
\end{tabular}
\end{table}

\begin{figure}[htbp]
\centering
\includegraphics[width=15cm,height=6cm]{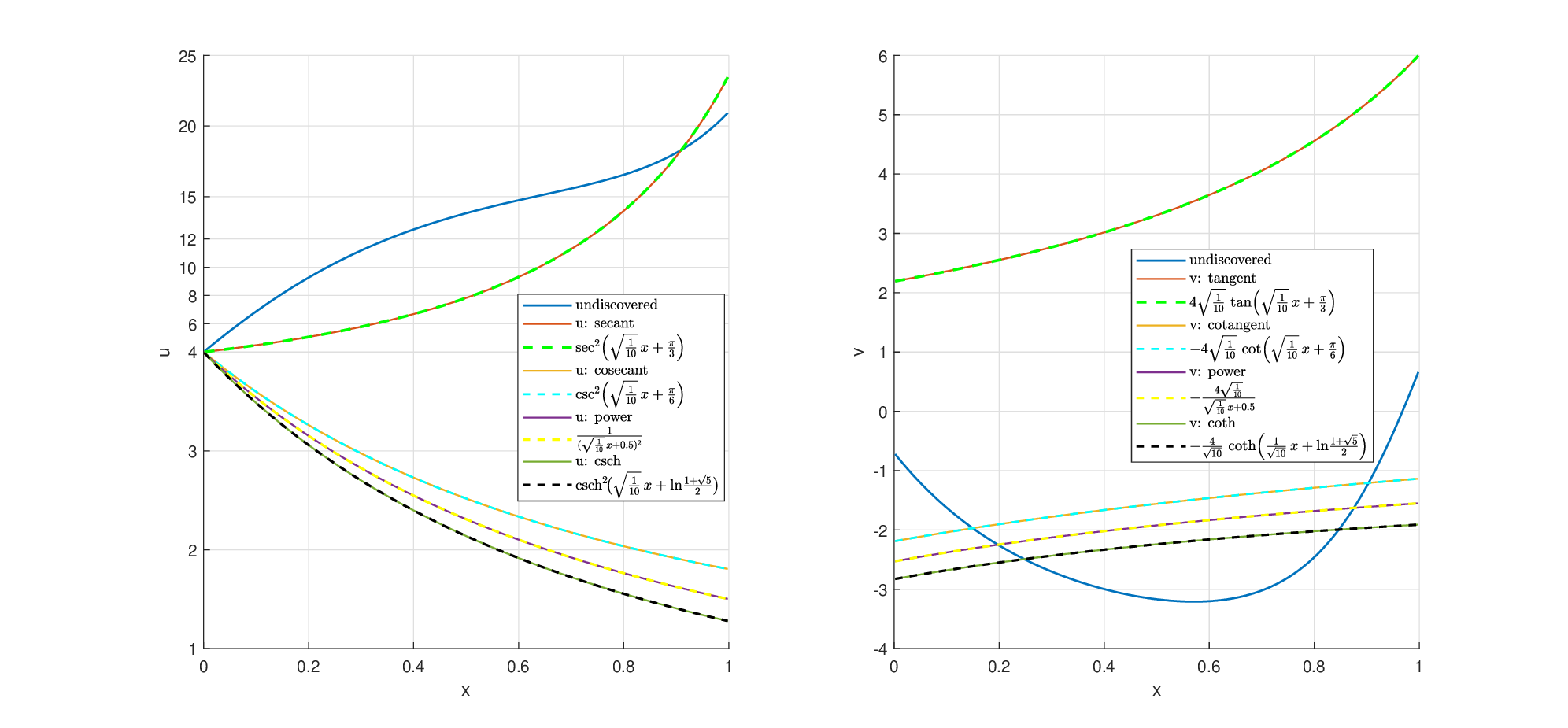}
\caption{Numerical solutions to \eqref{Tmodel} at $t=2$, with parameters in \eqref{parameters}, boundary values in Tables \ref{table-3}, initial conditions in Table \ref{table-4}, and comparison with analytic baselines.}
\label{fig:myplot}  
\end{figure}

Numerical solutions are computed up to time $t = 2$. As shown in Figure \ref{fig:myplot}, all solutions converge to steady states as time evolves. Except the solid blue curves, the other solid curves match perfectly with the corresponding explicit steady-state solutions (dotted lines), validating both the analytical results and the numerical method. This agreement highlights the robustness and applicability of the explicit solutions within their defined parametric regimes.

In contrast, the solid blue curve, obtained under boundary conditions that violate the constraints of Section 2, exhibits a qualitatively different structure, featuring both convex and concave segments in the bulk. This morphology stands in sharp contrast to the uniformly convex profiles observed in the top four cases. Such a deviation not only confirms the sensitivity of the steady-state structure to boundary conditions but also suggests the existence of richer and more complex solution behaviors outside the currently characterized regimes.

The existence of non-monotonic and hybrid convex-concave profiles under more general boundary conditions poses significant challenges for analytical treatment. It also emphasizes the need to extend the current solution framework to encompass a broader class of boundary data. Identifying explicit solutions to system \eqref{ss} under general Dirichlet boundary conditions \eqref{BC0}, particularly those yielding non-standard spatial morphologies, remains an important and non-trivial task for future work. A deeper understanding of such patterns could provide valuable insights into the nonlinear coupling and boundary-driven structuring in chemotactic and related systems.


\section{Conclusions and Outlook}

We have constructed three distinct classes of explicit steady-state solutions for the transformed chemotaxis system \eqref{Tmodel} under the parameter regime \eqref{para}. Using energy methods, we further established the asymptotic stability of these solutions under mild smallness conditions on the initial perturbations. However, it should be noted that the analysis presented in this work is conducted under the specific parameter choice $\mathsf{m} = 1$. The case where $\mathsf{m} \neq 1$ is currently under investigation and will be addressed in our subsequent work.

Beyond theoretical interest, these explicit solutions, together with their accompanying stability analysis, provide a foundational framework for understanding and manipulating chemotactic behavior in more complex, application-oriented settings. In particular, two promising research directions emerge from this work: multi-domain problems and optimal control applications. Both avenues build directly on our precise characterization of the base model in confined domains and offer considerable potential for developing new mathematical tools to address real-world challenges in tissue engineering and precision medicine. These directions represent critical pathways for advancing chemotaxis models from theoretical description toward active intervention. The steady states we have rigorously derived serve as reliable computational benchmarks and a solid theoretical starting point for such further investigations. Future work will focus on developing rigorously verified numerical schemes for multi-domain systems and deriving first-order optimality conditions for the corresponding optimal control problems. Through close collaboration with experimental biologists and engineers, these mathematical tools hold transformative potential for applications in regenerative medicine, cancer therapy, and biotechnological design.


\section*{Acknowledgement}
Support of this work came in part from the National Natural Science Foundation of China No.\,12171116 and No.\,12571527 (L. Xue), and Fundamental Research Funds for the Central Universities of China No.\,3072020CFT2402 (L. Xue), No.\,3072022TS2404, and No.\,3072024WD2401 (K. Zhao).



\begin{thebibliography}{99}

\bibitem{WF} D. Algom, The Weber-Fechner law: a misnomer that persists but that should go away, {\it Psychological
Rev.}, 2021, DOI: http://dx.doi.org/10.1037/rev0000278.




\bibitem{Spike} J.A. Carrillo, J. Li and Z. Wang, Boundary spike-layer solutions of the singular Keller–Segel system: existence and stability, {\it Proc. London Math. Soc.}, {\bf 122} (2021), 42--68.

\bibitem{mTW-1} M. Chae and K. Choi, Nonlinear stability of planar traveling waves in a chemotaxis model of tumor angiogenesis with chemical diffusion, {\it J. Diff. Equ.}, {\bf 268} (2020), 3449--3496.

\bibitem{mTW-2} M. Chae, K. Choi, K. Kang and J. Lee, Stability of planar traveling waves in a Keller-Segel equation on an infinite strip domain, {\it J. Diff. Equ.}, {\bf 265} (2018), 237--279.

\bibitem{TW1} K. Choi, M.-J. Kang, Y.-S. Kwon and A. Vasseur, Contraction for large perturbations of traveling waves in a hyperbolic-parabolic system arising from a chemotaxis model, {\it Math. Models Meth. Appl. Sci.}, {\bf 30} (2020), 387--437.

\bibitem{FTWZZ} Z. Feng, Y. Tang, W. Wang, K. Zhao and N. Zhu, On a chemotaxis model with singular sensitivity: convergence rate towards spiky steady state, {\it Disc. Cont. Dyn. Syst.}, {\bf 45} (2025), 2281--2301.

\bibitem{FXXZ} Z. Feng, J. Xu, L. Xue and K. Zhao, Initial and boundary value problem for a system of balance laws from chemotaxis: global dynamics and diffusivity limit, {\it Ann. Appl. Math.}, {\bf 37} (2021), 61--110.

\bibitem{FZZ} Z. Feng, K. Zhao and S. Zhou, Existence and stability of boundary spike layer solutions of an attractive chemotaxis model with singular sensitivity and nonlinear consumption rate of chemical stimuli, {\it Physica D}, {\bf 471} (2025), Article No. 134429.

\bibitem{FFH} M.A. Fontelos, A. Friedman and B. Hu, Mathematical analysis of a model for the initiation of angiogenesis, {\it SIAM J. Math. Anal.}, {\bf 33} (2002), 1330--1355.

\bibitem{FZ} P. Fuster-Aguilera and K. Zhao, Dynamical behavior of a logarithmically sensitive chemotaxis model under time-dependent  boundary conditions, {\it European J. Appl. Math.}, {\bf 36} (2025), 638--664. 

\bibitem{L1} P. Fuster-Aguilera, V.R. Martinez and K. Zhao, A PDE model for chemotaxis with logarithmic sensitivity and logistic growth, {\it Contemporary Research in Mathematical Biology: Modeling, Computation and Analysis}, Contemporary Mathematics and its Applications: Monographs, Expositions and Lecture Notes, DOI: https://doi.org/10.1142/12639.

\bibitem{GXZZ} J. Guo, J. Xiao, H. Zhao and C. Zhu, Global solutions to a hyperbolic-parabolic coupled system with large initial data, {\it Acta Math. Sci. Ser. B (Engl. Ed.)}, {\bf 29} (2009), 629--641.

\bibitem{DL3} Q. Hou and Z. Wang, Convergence of boundary layers for the Keller-Segel system with singular sensitivity in the half-plane, {\it J. Math. Pures. Appl.}, {\bf 130} (2019), 251--287.

\bibitem{DL2} Q. Hou, C. Liu, Y. Wang and Z. Wang, Stability of boundary layers for a viscous hyperbolic system arising from chemotaxis: one dimensional case, {\it SIAM J. Math. Anal.}, {\bf 50} (2018), 3058--3091.

\bibitem{DL1} Q. Hou, Z. Wang and K. Zhao, Boundary layer problem on a hyperbolic system arising from chemotaxis, {\it J. Differential Equations}, {\bf 261} (2016), 5035--5070.

\bibitem{TW2} H. Jin, J. Li and Z. Wang, Asymptotic stability of traveling waves of a chemotaxis model with singular sensitivity, {\it J. Differential Equations}, {\bf 255} (2013), 193--219.

\bibitem{KS3} E.F. Keller and L.A. Segel, Traveling bands of chemotactic bacteria: a theoretical analysis, {\it J. Theo. Biol.}, {\bf 26} (1971), 235--248.

\bibitem{LS} H.A. Levine and B.D. Sleeman, A system of reaction diffusion equations arising in the theory of reinforced random walks, {\it SIAM J. Appl. Math.}, {\bf 57} (1997), 683--730.

\bibitem{LSN} H.A. Levine, B.D. Sleeman and M. Nilsen-Hamilton, A mathematical model for the roles of pericytes and macrophages in the initiation of angiogenesis. i. the role of protease inhibitors, {\it Math. Biosci.}, {\bf 168} (2000), 77--115.

\bibitem{1d1} D. Li, R. Pan and K. Zhao, Quantitative decay of a one-dimensional hybrid chemotaxis model with large data, {\it Nonlinearity}, {\bf 28} (2015), 2181--2210.

\bibitem{1d2} H. Li and K. Zhao, Initial-boundary value problems for a system of hyperbolic balance laws arising from chemotaxis, {\it J. Differential Equations}, {\bf 258} (2015), 302--338.

\bibitem{TW3} J. Li. T. Li and Z. Wang, Stability of traveling waves of the Keller-Segel system with logarithmic sensitivity, {\it Math. Models Meth. Appl. Sci.}, {\bf 24} (2014), 2819--2849.

\bibitem{DL} C. Deng and T. Li, Well-posedness of a 3D parabolic-hyperbolic Keller-Segel system in the Sobolev space framework, {\it J. Differential Equations}, {\bf 257} (2014), 1311--1332.

\bibitem{1d3} T. Li, R. Pan and K. Zhao, Global dynamics of a hyperbolic-parabolic model arising from chemotaxis, {\it SIAM J. Appl. Math.}, {\bf 72} (2012), 417--443.

\bibitem{MD3} T. Li, D. Wang, F. Wang, Z. Wang and K. Zhao, Large time behavior and diffusion limit for a system of balance laws from chemotaxis in multi-dimensions, {\it Comm. Math. Sci.}, {\bf 19} (2021), 229--272.

\bibitem{TW4} T. Li and Z. Wang, Nonlinear stability of traveling waves to a hyperbolic-parabolic system modeling chemotaxis, {\it SIAM J. Appl. Math.}, {\bf 7} (2009), 1522--1541.

\bibitem{TW5} T. Li, and Z. Wang, Nonlinear stability of large amplitude viscous shock waves of a generalized hyperbolic-parabolic system arising in chemotaxis, {\it Math. Models Meth. Appl. Sci.}, {\bf 20} (2010), 1967--1998.

\bibitem{TW6} T. Li,  and Z. Wang, Asymptotic nonlinear stability of traveling waves to conservation laws arising from chemotaxis, {\it J. Differential Equations}, {\bf 250} (2011), 1310--1333.

\bibitem{TW7} T. Li, and Z. Wang, Steadily propagating waves of a chemotaxis model, {\it Math. Biosci.}, {\bf 240} (2012), 161--168.

\bibitem{1d4} V. Martinez, Z. Wang and K. Zhao, Asymptotic and viscous stability of large-amplitude solutions of a hyperbolic system arising from biology, {\it Indiana Univ. Math. J.}, {\bf 67} (2018), 1383--1424.

\bibitem{OS} H. Othmer and A. Stevens, Aggregation, blowup and collapse: The ABC's of taxis in reinforced random walks, {\it SIAM J. Appl. Math.}, {\bf 57} (1997), 1044--1081.

\bibitem{TW8} H. Peng and Z. Wang, Nonlinear stability of strong traveling waves for the singular Keller-Segel system with large perturbations, {\it J. Differential Equations}, {\bf 265} (2018), 2577--2613.

\bibitem{1d5} H. Peng, Z. Wang, K. Zhao and C. Zhu, Boundary layers and stabilization of the singular Keller-Segel model, {\it Kinet. Relat. Models}, {\bf 11} (2018), 1085--1123.

\bibitem{PWZ} H. Peng, H. Wen and C. Zhu. Global well-posedness and zero diffusion limit of classical solutions to 3D conservation laws arising in chemotaxis. {\it Z. Angew. Math. Phys.} {\bf 65} (2014): 1167--1188.

\bibitem{MD1} L. Rebholz, D. Wang, Z. Wang, C. Zerfas and K. Zhao, Initial boundary value problems for a system of parabolic conservation laws arising from chemotaxis in multi-dimensions, {\it Disc. Conti. Dyn. Sys.}, {\bf 39} (2019), 3789--3838.

\bibitem{1d6} Y. Tao, L. Wang and Z. Wang, Large-time behavior of a parabolic-parabolic chemotaxis model with logarithmic sensitivity in one dimension, {\it Disc. Cont. Dyn. Syst., Ser. B}, {\bf 18} (2013), 821--845.

\bibitem{MD2} D. Wang, Z. Wang and K. Zhao, Cauchy problem of a system of parabolic conservation laws arising from the singular Keller-Segel model in multi-dimensions, {\it Indiana Univ. Math. J.}, {\bf 70} (2021), 1--47.

\bibitem{TW9} Z. Wang, Mathematics of traveling waves in chemotaxis, {\it Disc. Cont. Dyn. Syst. Ser. B}, {\bf 18} (2013), 601--641.

\bibitem{WXY} Z. Wang, Z. Xiang and P. Yu. Asymptotic dynamics on a singular chemotaxis system modeling onset of tumor angiogenesis. {\it J. Differential Equations} {\bf 260} (2016): 2225--2258.

\bibitem{1d7} Z. Wang and K. Zhao, Global dynamics and diffusion limit of a one-dimensional repulsive chemotaxis model, {\it Commun. Pure Appl. Anal.}, {\bf 12} (2013), 3027--3046.

\bibitem{Winkler-large-1} M. Winkler, The two-dimensional Keller–Segel system with singular sensitivity and signal absorption: global large-data solutions and their relaxation properties, {\it Math. Models Methods Appl. Sci.}, {\bf 26} (2016), 987--1024.

\bibitem{Winkler-large-2} M. Winkler, Renormalized radial large-data solutions to the higher-dimensional Keller-Segel system with singular sensitivity and signal absorption, {\it J. Differential Equations}, {\bf 264} (2018), 2310--2350.

\bibitem{ZZ} M. Zhang and C. Zhu, Global existence of solutions to a hyperbolic-parabolic system, {\it Proc. Amer. Math. Soc.}, {\bf 135} (2006), 1017--1027.

\bibitem{ZZMZ} N. Zhu, Z. Liu, V. Martinez and K. Zhao, Global Cauchy problem of a system of parabolic conservation laws arising from a Keller-Segel type chemotaxis model, {\it SIAM J. Math. Anal.}, {\bf 50} (2018), 5380--5425.

\bibitem{ZLWZ} N. Zhu, Z. Liu, F. Wang and K. Zhao, Asymptotic dynamics of a system of conservation laws from chemotaxis, {\it Disc. Cont. Dyn. Syst.}, {\bf 41} (2021), 813--847.


\end{thebibliography}
\end{document}